\documentclass[a4paper,11pt]{article}
\usepackage[english]{babel}
\usepackage[latin1]{inputenc}
\usepackage{amsmath}
\usepackage{amsfonts}
\usepackage{amssymb}
\usepackage{epsfig}
\usepackage{amsopn}
\usepackage{amsthm}
\usepackage{color}
\usepackage{graphicx}
\newtheorem{theorem}{Theorem}

\newtheorem*{theorem*}{Theorem}
\newtheorem*{lemma*}{Lemma}
\newtheorem*{remark*}{Remark}
\newtheorem*{definition*}{Definition}
\newtheorem*{proposition*}{Proposition}
\newtheorem*{corollary*}{Corollary}
\numberwithin{equation}{section} \numberwithin{theorem}{section}
\numberwithin{proposition}{section} \numberwithin{lemma}{section}
\numberwithin{definition}{section} \numberwithin{corollary}{section}
\newcommand{\real}{\mathbb{R}}

\def\supp{\mathrm{supp}} 

\def\qed{\,\unskip\kern 6pt \penalty 500
\raise -2pt\hbox{\vrule \vbox to8pt{\hrule width 6pt
\vfill\hrule}\vrule}\par}
\newcommand{\beqn}{\begin{equation}}
\newcommand{\eeqn}{\end{equation}}
\newcommand{\bear}{\begin{eqnarray}}
\newcommand{\eear}{\end{eqnarray}}
\newcommand{\bean}{\begin{eqnarray*}}
\newcommand{\eean}{\end{eqnarray*}}
\usepackage{subfigure}
\usepackage{enumerate}
\begin{document}

\title{\huge \bf Equivalence and finite time blow-up of solutions and interfaces for two nonlinear diffusion equations}

\author{\Large Benito Hern\'andez-Bermejo\footnote{\textit{e-mail:} benito.hernandez@urjc.es}, \ Razvan Gabriel Iagar\footnote{\textit{e-mail:}
razvan.iagar@urjc.es}, \
Pilar R. Gordoa\footnote{\textit{e-mail:} pilar.gordoa@urjc.es},
\\ \Large Andrew Pickering\footnote{\textit{e-mail:} andrew.pickering@urjc.es}, \ Ariel S\'anchez\footnote{\textit{e-mail:} ariel.sanchez@urjc.es} \\ [5mm]
Universidad Rey Juan Carlos de Madrid, 28933, M\'ostoles, Madrid, Spain
}
\date{}
\maketitle

\begin{abstract}
In this work, we construct a transformation between the solutions to the following reaction-convection-diffusion equation
$$
\partial_t u=(u^m)_{xx}+a(x)(u^m)_x+b(x)u^m,
$$
posed for $x\in\real$, $t\geq0$ and $m>1$, where $a$, $b$ are two continuous real functions, and the solutions to the nonhomogeneous diffusion equation of porous medium type
$$
f(y)\partial_{\tau}\theta=(\theta^m)_{yy},
$$
posed in the half-line $y\in[0,\infty)$ with $\tau\geq0$, $m>1$ and suitable density functions $f(y)$. We apply this correspondence to the case
of constant coefficients $a(x)=1$ and $b(x)=K>0$. For this case, we prove that compactly supported solutions to the first equation blow up in finite time, together with their interfaces, as $x\to-\infty$. We then establish the large time behavior of solutions to a homogeneous Dirichlet problem associated to the first equation on a bounded interval. We also prove a finite time blow-up of the interfaces for compactly supported solutions to the second equation when $f(y)=y^{-\gamma}$ with $\gamma>2$.
\end{abstract}

\


\smallskip


\section{Introduction}

The goal of this paper is to establish an explicit correspondence between two equations with non-constant coefficients of different types, namely one involving diffusion, convection and reaction terms and the other of nonhomogeneous nonlinear diffusion type. More precisely, let us consider on the one hand the equation
\begin{equation}\label{eq1.gen}
u_t=(u^m)_{xx}+a(x)(u^m)_x+b(x)u^m,
\end{equation}
posed for $(x,t)\in\real\times(0,T)$, with $m>1$ and $a(x)$, $b(x)$ continuous real functions and where, as usual, the subscript indicates derivative. This equation presents three different effects: diffusion, convection and reaction, and the main interest in its study is related to how these terms compete, whether blow-up in finite time occurs or not and, if it occurs, in which sets. In the applications we consider the following interesting particular case of equation \eqref{eq1.gen}
\begin{equation}\label{eq1}
u_t=(u^m)_{xx}+(u^m)_x+Ku^m,
\end{equation}
with $K>0$.

On the other hand, let us consider the following nonhomogeneous equation of porous medium type:
\begin{equation}\label{eq2}
f(y)\theta_{\tau}=(\theta^m)_{yy},
\end{equation}
posed for $(y,\tau)\in[0,\infty)\times(0,T)$, with $m>1$ and a suitable function $f(y)$. Starting from the series of papers by Kamin and Rosenau \cite{KR81, KR82, KR83}, where equations of the form \eqref{eq2} were proposed as models for the thermal propagation by radiation in nonhomogeneous plasma, an extensive study of equation \eqref{eq2} for suitable density functions $f(y)$ has been developed over the last three decades by researchers such as Eidus, Galaktionov, Kamin, Kersner, King, Reyes, Tesei, V\'azquez et. al. and many interesting properties have been discovered. Most of this development has been devoted to densities of the form
\begin{equation}\label{powerlike.dens}
f(y)\sim |y|^{-\gamma}, \quad {\rm as} \ |y|\to\infty, \quad f(0)>0,
\end{equation}
where qualitative properties and large time behavior of solutions has been understood both for $\gamma\in(0,2)$ and for $\gamma>2$, see for example \cite{GKKV04, RV06, RV09, KRV10} and references therein. In particular, it has been noticed that $\gamma=2$ is a critical value that separates two different types of behavior for equation \eqref{eq2}. Thus, while for $\gamma\in(0,2)$ the dynamics of the equation is very similar to that of the standard porous medium equation (that is, $\gamma=0$), in the range $\gamma>2$ the dynamics proves to be very different \cite{KRV10}. Special solutions of self-similar type no longer exist for $\gamma>2$ and explaining the large time behavior leads to more complex considerations. In parallel to this study, the pure power densities
\begin{equation}\label{power.dens}
f(y)=|y|^{-\gamma}, \quad \gamma>0,
\end{equation}
were also considered, as it had been noticed that the profiles for asymptotic behavior to equation \eqref{eq2} with densities as in \eqref{powerlike.dens} come from special solutions to equation \eqref{eq2} with densities as in \eqref{power.dens}. These densities are more complicated to study due to the singularity at the origin, and they were considered first in dimension one in \cite{GK96, GKKV04} to establish self-similar solutions and to show the blow-up of interfaces in finite time, a result that will be very useful in the sequel. Deeper analysis, including a thorough study of well posedness, is performed in \cite{KT04} and later for $\gamma>2$ and general dimension $N\geq3$ in \cite[Section 6]{KRV10}. More recently, the critical exponent $\gamma=2$ was considered and in \cite{IS14, IS16} the large time behavior to equation \eqref{eq2} with $f(y)=|y|^{-2}$ is established. We also mention here the paper \cite{IRS13} where a number of mappings between solutions to equation \eqref{eq2} with $f(y)$ as in \eqref{power.dens} for different values of $\gamma$ have been obtained. Some of these correspondences will be very useful in the present work. Finally, equation \eqref{eq2} with exponential density function $f(y)=e^{-y}$ was proposed in \cite{GK96} where a study of its self-similar solutions is performed and blow-up of the interfaces in finite time for compactly supported data is proved.

Coming back to equation \eqref{eq1}, much less is known. This equation is considered in \cite{dPS00}, where solutions in the form of traveling waves are constructed both for $K<0$ and for $K>0$, with the form
\begin{equation}\label{TW}
u(x,t)=\Phi(x-st), \quad s>0, \ \lim\limits_{z\to\infty}\Phi(z)=0,
\end{equation}
and these solutions are compactly supported to the right (have an interface) for $K\in(0,1/4)$. equation \eqref{eq1} with $K>0$ is also a particular case of a more general equation considered by Suzuki in \cite{Su98}, where it is proved that the exponent $m$ in the reaction term lies below the Fujita-type exponent $m+2$ and there are no nontrivial global solutions, that is, all the solutions to equation \eqref{eq1} with $K>0$ blow up in finite time. On the contrary, equation \eqref{eq1} with $K<0$ (that is, when we are dealing with an absorption effect instead of reaction) can be easily mapped by a straightforward change of variable into the Fisher-KPP type equation
$$
u_t=(u^m)_{xx}+a(u^m)_x+u-u^m,
$$
whose qualitative theory and large time behavior are known (see for example \cite{Ro02, KR04}). This is why we focus in the part of applications in our paper to the reaction range $K>0$ and show how and where finite time blow-up occurs.

\medskip

\noindent \textbf{Description of the results and structure of the paper.} The core of the paper consists in pointing out a \emph{new transformation} mapping solutions to equation \eqref{eq1.gen} (and in particular, solutions to equation \eqref{eq1}) onto solutions to suitable nonhomogeneous equations of the form \eqref{eq2} with suitable density functions $f(y)$. This correspondence strongly generalizes previously established mappings that were considered in works such as \cite{Ki93, GPV00, Ro02, KR04, CSR08, IS16} only for some particular cases of the two equations involved. The transformation is constructed in Section \ref{sec.transf}. The rest of the paper is devoted to some \emph{applications of the transformation} with emphasis on the solutions to equation \eqref{eq1} for $K>0$.

Our first and most interesting application is to study \emph{how finite time blow-up occurs} for solutions to \eqref{eq1} when the coefficient of the reaction term is $0<K\leq1/4$. We define the \emph{blow-up time $T$} of a generic solution $u$ to equation \eqref{eq1} as the smallest time $T\in(0,\infty)$ such that $u(t)\in L^{\infty}(\real)$ for any $t\in(0,T)$ but $u(T)\not\in L^{\infty}(\real)$. We prove in Section \ref{sec.bu} that solutions to \eqref{eq1} with $0<K\leq1/4$ and with compactly supported initial condition $u_0(x)=u(x,0)$ blow up in finite time in a rather special way: at their blow-up time $T\in(0,\infty)$, $u(x,T)$ remains bounded at any finite point $x$. However, they still blow up at $t=T$, but only on curves $x(t)$ depending on $t$ such that $x(t)\to-\infty$ as $t\to T$. This phenomenon is known in literature as \emph{blow-up at (space) infinity} (in our case at $-\infty$). It seems to have been considered for the first time by Lacey \cite{La84}, and some other cases where it has been established, either for semilinear reaction-diffusion equations with rather large initial conditions or for quasilinear reaction-diffusion equations with weighted reaction, can be found in \cite{GU05, GU06, IS19}. The precise statement, which includes more initial data than the compactly supported ones, is given as Theorem \ref{th.bu} at the beginning of Section \ref{sec.bu}, and its proof is given in the Subsections \ref{sec.Ksmall} and \ref{sec.Kequal}.

For $K>1/4$, the transformation changes and its outcome limits us to consider the homogeneous Dirichlet problem for equation \eqref{eq1} posed in a bounded domain. In that case, we are able to apply the transformation and some classical results on equation \eqref{eq2} to establish \emph{the large time behavior} of solutions to equation \eqref{eq1} in a special bounded interval. This is the subject of Section \ref{sec.Klarge}.

Since the correspondence introduced in Section \ref{sec.transf} works in both direction, we are also able to apply it in order to improve on the theory of equation \eqref{eq2}. More precisely, we show in Section \ref{sec.app} that solutions to equation \eqref{eq2} with density function $f(y)=y^{-\gamma}$ and $\gamma>2$, with compactly supported initial conditions inside the half-line $(0,\infty)$, present a \emph{blow-up of the interface} in finite time, completing thus the results in \cite{GK96} to different types of initial data. This means that there exists $T\in(0,\infty)$ such that the solution $\theta(y,\tau)$ to equation \eqref{eq2} remains compactly supported for $\tau\in(0,T)$ but develops a tail as $y\to\infty$ at time $\tau=T$. We make this precise in Theorem \ref{th.nohom}.

We close the paper with an Appendix where some exact solutions are given and some natural extensions and open problems are discussed.

\section{The transformation}\label{sec.transf}

In this section we establish the mapping between solutions to equation \eqref{eq1} and equation \eqref{eq2}. In order to introduce the transformation in its most abstract and general form, let us begin with the more general reaction-convection-diffusion equation with nonconstant coefficients equation \eqref{eq1.gen},
where $a$, $b:\real\mapsto\real$ are two continuous functions. If $w$ and $v$ are two linearly independent solutions to the ordinary differential equation
\begin{equation}\label{ODE}
\sigma_{xx}+a(x)\sigma_x+b(x)\sigma=0,
\end{equation}
which is the stationary part of equation \eqref{eq1.gen}, introducing the Wronskian
$$
W(x)=w(x)v'(x)-w'(x)v(x),
$$
we can write equation \eqref{eq1.gen} in the following form
\begin{equation}\label{interm1}
\frac{w^3}{W^2}u_t=\frac{w^2}{W}\frac{\partial}{\partial x}\left(\frac{w^2}{W}\frac{\partial}{\partial x}\left(\frac{u^m}{w}\right)\right).
\end{equation}
Setting now
\begin{equation}\label{change}
\theta(y,t)=\frac{u(x,t)}{(w(x))^{1/m}}, \qquad y=\frac{v(x)}{w(x)},
\end{equation}
we readily notice that
$$
\frac{\partial}{\partial y}=\frac{w^2}{W}\frac{\partial}{\partial x},
$$
whence after straightforward calculations \eqref{interm1} writes as equation \eqref{eq2} in terms of $\theta$ and independent variables $(y,t)$, with
\begin{equation}\label{interm2}
f(y)=\frac{w(x)^{(3m+1)/m}}{W(x)^2}, \quad y=\frac{v(x)}{w(x)}.
\end{equation}

\noindent \textbf{The case of constant coefficients and equation \eqref{eq1}.} Let us particularize now the transformation to the case of constant coefficients $a(x)=a\neq0$, $b(x)=b>0$. By a standard rescaling and taking into account the change of variable $x\mapsto-x$ for the case $a<0$, we are easily mapped into equation \eqref{eq1}, that is, the case $a=1$, $b=K>0$. Thus, for simplicity we develop our applications for equation \eqref{eq1}. In this case \eqref{ODE} can be solved. Indeed, its characteristic polynomial is
$$
P(\lambda)=\lambda^2+\lambda+K,
$$
with roots
\begin{equation}\label{lambdas}
\lambda_1=\frac{-1-\sqrt{1-4K}}{2}, \qquad \lambda_2=\frac{-1+\sqrt{1-4K}}{2}.
\end{equation}
We split the analysis into three cases, according to the sign of $1-4K$.

\medskip

\noindent \textbf{Case 1. $-\infty<K<1/4$.} The two linearly independent solutions and its Wronskian become
$$
w(x)=e^{\lambda_1x}, \quad v(x)=e^{\lambda_2x}, \quad W(x)=(\lambda_2-\lambda_1)e^{(\lambda_1+\lambda_2)x}.
$$
The change of variables \eqref{change} writes in this particular case
\begin{equation}\label{change.Ksmall}
\theta(y,t)=\frac{u(x,t)}{e^{\lambda_1x/m}}, \quad y=e^{(\lambda_2-\lambda_1)x},
\end{equation}
and we derive from \eqref{interm2} that
\begin{equation*}
\begin{split}
f(y)&=\frac{e^{(3m+1)\lambda_1x/m}}{(\lambda_2-\lambda_1)^2e^{2(\lambda_1+\lambda_2)x}}=\frac{1}{(\lambda_2-\lambda_1)^2}e^{((m+1)\lambda_1-2m\lambda_2)x/m}\\
&=\frac{1}{(\lambda_2-\lambda_1)^2}y^{((m+1)\lambda_1-2m\lambda_2)/m(\lambda_2-\lambda_1)}.
\end{split}
\end{equation*}
Making a further change of the time variable by letting $\tau=(\lambda_2-\lambda_1)^2t$, we get
\begin{equation}\label{eq2.Ksmall}
y^{-\gamma}\theta_{\tau}=(\theta^m)_{yy}, \quad \gamma=\frac{2m\lambda_2-(m+1)\lambda_1}{m(\lambda_2-\lambda_1)}=\frac{1}{2}\left[\frac{3m+1}{m}-\frac{m-1}{m\sqrt{1-4K}}\right].
\end{equation}
\noindent \textbf{Remark 1.} We can equate $K$ from the expression of $\gamma$ in \eqref{eq2.Ksmall} to obtain
\begin{equation}\label{interm9}
K=\frac{m(\gamma-2)(m\gamma-m-1)}{(2m\gamma-(3m+1))^2},
\end{equation}
hence $K\geq0$ for either $\gamma\in(0,(m+1)/m)$ or $\gamma>2$. Let us notice first that we get $\gamma=0$ (that is, the standard porous medium equation) for $K=2m(m+1)/(3m+1)^2$. On the other hand, the transformation leads to equation \eqref{eq1} with $K=0$ (that is, without reaction term) for $\gamma=(m+1)/m$, which is nothing else that the particular value in dimension $N=1$ of the following critical exponent
$$
m=m_{c,\gamma}:=\frac{N-2}{N-\gamma}
$$
introduced and studied in \cite{IRS13}. In this case, the transformation \eqref{change.Ksmall} becomes a transformation between the nonhomogeneous equation \eqref{eq2} with $m=m_{c,\gamma}$ and \eqref{eq1} with $K=0$, which generalizes the well-established transformation for the standard fast diffusion equation with critical exponent
$$
u_t=(u^m)_{xx}, \quad m=m_c:=\frac{(N-2)_+}{N}
$$
established formally by King \cite{Ki93} and used by Galaktionov, Peletier and V\'azquez \cite{GPV00} in their study of the qualitative properties of the critical fast diffusion. Let us also notice that we can obtain $K=0$ by letting $\gamma=2$. This is another particular case of the transformation that is already known, and in particular it has been used as a main tool in establishing the large time behavior for equation \eqref{eq2} with $f(y)=y^{-2}$ in \cite{IS16}.

\medskip

\noindent \textbf{Remark 2.} It may look surprising that we have a correspondence between an autonomous equation \eqref{eq1} and a non-autonomous equation \eqref{eq2}, as the former is invariant to space translations while the latter is not. However, using the change of variables \eqref{change.Ksmall} one can readily check that the invariance to translations in space for equation \eqref{eq1} is mapped into the following rescaling of equation \eqref{eq2}
\begin{equation}\label{resc.nohom}
\theta(y,\tau)=y_0^{\alpha}\overline{\theta}(yy_0,\tau), \quad \alpha=\frac{\gamma-2}{m-1}.
\end{equation}

\medskip

\noindent \textbf{Case 2. $K=1/4$.} In this case we have $\lambda_1=\lambda_2=-1/2$, hence
$$
w(x)=e^{-x/2}, \quad v(x)=xe^{-x/2}, \quad W(x)=e^{-x}.
$$
The change of variables \eqref{change} writes in this particular case (in the same independent variables $(x,t)$)
\begin{equation}\label{change.Kequal}
\theta(x,t)=e^{x/2m}u(x,t),
\end{equation}
thus the nonhomogeneous equation becomes
$$
f(x)\theta_t=(\theta^m)_{xx}, \quad f(x)=e^{(m-1)x/2m}.
$$
Letting furthermore
$$
y=-\frac{m-1}{2m}x, \quad \tau=\frac{(m-1)^2}{4m^2}t,
$$
we find our final nonhomogeneous equation in this case
\begin{equation}\label{eq2.Kequal}
e^{-y}\theta_{\tau}=(\theta^m)_{yy},
\end{equation}
which has been analyzed in \cite{GK96}. We will come back later to this study and its applications to equation \eqref{eq1}

\medskip

\noindent \textbf{Case 3. $K>1/4$.} In this case the roots in \eqref{lambdas} are complex and become
$$
\lambda_1=\frac{-1-i\sqrt{4K-1}}{2}, \quad \lambda_2=\frac{-1+i\sqrt{4K-1}}{2},
$$
hence
$$
w(x)=e^{-x/2}\cos\left(\frac{\sqrt{4K-1}}{2}x\right), \quad v(x)=e^{-x/2}\sin\left(\frac{\sqrt{4K-1}}{2}x\right)
$$
and the Wronskian is given by
$$
W(x)=\frac{\sqrt{4K-1}}{2}e^{-x}.
$$
The change of variables \eqref{change} writes in this particular case
\begin{equation}\label{change.Kbig}
\theta(y,t)=e^{x/2m}\frac{u(x,t)}{\cos(\sqrt{4K-1}x/2)^{1/m}}, \quad y=\frac{v(x)}{w(x)}=\tan\left(\frac{\sqrt{4K-1}}{2}x\right),
\end{equation}
thus we obtain from \eqref{interm2} that
\begin{equation}\label{interm4}
\begin{split}
f(y)&=e^{(m-1)x/2m}\frac{4}{4K-1}\cos\left(\frac{\sqrt{4K-1}}{2}x\right)^{(3m+1)/m}\\
&=\frac{4}{4K-1}e^{(m-1)\arctan(y)/m\sqrt{4K-1}}\left(\frac{1}{(1+y^2)}\right)^{(3m+1)/2m}.
\end{split}
\end{equation}
Let us notice that in performing the previous change of variables, we have to restrict ourselves to the bounded domain
$$
x\in\left(\frac{-\pi}{\sqrt{4K-1}},\frac{\pi}{\sqrt{4K-1}}\right),
$$
which allows for the new variable $y$ to be well defined. Thus, when $K>1/4$, the transformation loses some information as in this range of $K$, we have to consider equation \eqref{eq1} in a bounded domain (which shrinks as $K$ increases). However, we arrive to the nonhomogeneous equation \eqref{eq2}, with the specific function $f(y)$ given in \eqref{interm4}, which can be seen as
\begin{equation}\label{eq2.Kbig}
f(y)\theta_t=(\theta^m)_{yy}, \quad f(y)\sim|y|^{-(3m+1)/m} \ {\rm as} \ |y|\to\infty.
\end{equation}

\noindent \textbf{Remark.} It is not just a technical fact that the value $K=1/4$ plays such a role in the analysis of equation \eqref{eq1}. As also explained in the Introduction, it has been proved in \cite{dPS00} that traveling waves exist for $K<1/4$ (and in particular they are compactly supported for $0<K<1/4$), but for $K\geq1/4$ they cease to exist.

\section{Blow-up at $-\infty$ for $0<K\leq1/4$}\label{sec.bu}

As explained in the Introduction, from now on we assume that $K>0$ and apply the transformations in Section \ref{sec.transf} to the study of the finite time blow-up phenomenon for equation \eqref{eq1}. We already know that for $K>0$, every nontrivial solution to equation \eqref{eq1} blows up in finite time \cite{Su98}, thus we will focus on how this blow-up occurs, that is, at which points and with which rate. In the present section, we consider $K>0$ sufficiently small, more precisely
$$
0<K\leq\frac{1}{4},
$$
corresponding via the transformation \eqref{change.Ksmall} to the equation \eqref{eq2.Ksmall} with $-\infty<\gamma<(m+1)/m$. The main result concerning finite time blow-up is the following
\begin{theorem}\label{th.bu}
Let $u$ be a solution to equation \eqref{eq1} with $K\in(0,1/4]$ and such that its initial condition $u_0(x)=u(x,0)$ is continuous, bounded and has a support of the form ${\rm supp}\,u_0\subset[A,\infty)$ for some $A\in\real$. Then the solution $u$ blows up in finite time together with its interface to the left. Moreover, the blow-up occurs at $-\infty$, that is, $u(x,t)$ remains bounded as $t\to T$ at any point $x\in\real$, but there exist curves $x(t)$ such that
$$
\lim\limits_{t\to T}x(t)=-\infty, \quad \lim\limits_{t\to T}u(x(t),t)=+\infty.
$$
\end{theorem}

\noindent \textbf{Remark.} Compactly supported initial data $u_0$ are included in Theorem \ref{th.bu} but the statement is more general.

The proof of Theorem \ref{th.bu} is based on the transformation introduced in Section \ref{sec.transf} and differs with respect to the value of $K$. This is why we divide it into two subsections.

\subsection{Proof of Theorem \ref{th.bu} for $0<K<1/4$}\label{sec.Ksmall}

We postpone for the moment the limit case $K=1/4$ and divide the analysis into three cases.

\medskip

\noindent \textbf{The particular case $K=2m(m+1)/(3m+1)^2$.} As we saw, this corresponds to $\gamma=0$, that is, \eqref{eq2.Ksmall} becomes the standard porous medium equation
\begin{equation}\label{PME}
\theta_{\tau}=(\theta^m)_{yy}.
\end{equation}
Let us recall that the transformation in this particular case has been noticed before in \cite{CSR08}. We begin with an explicit example which shows neatly how finite time blow-up of both interfaces and functions takes place. Let
\begin{equation}\label{trans.Baren}
B_{C,y_0}(y,\tau)=\tau^{-1/(m+1)}\left(C-k\frac{(y-y_0)^2}{\tau^{2/(m+1)}}\right)_{+}^{1/(m-1)}
\end{equation}
be a translated Barenblatt solution to \eqref{PME} \cite{VPME}. This solution has a left-interface at $y=0$ exactly for
$$
\tau_0=\left(\frac{ky_0^2}{C}\right)^{(m+1)/2}
$$
and ${\rm supp}\,B_{C,y_0}\subset(0,\infty)$ for $\tau\in(0,\tau_0)$. For $\tau=\tau_0$, the behavior near the interface of $B_{C,y_0}$ is given by
\begin{equation}\label{interm5}
B_{C,y_0}(y,\tau_0)=C^{1/(m-1)}\left(\frac{y(2y_0-y)}{y_0^2}\right)^{1/(m-1)}\sim(2C)^{1/(m-1)}\left(\frac{y}{y_0}\right)^{1/(m-1)}.
\end{equation}
Coming back to equation \eqref{eq1} by \eqref{change.Ksmall}, we find an explicit solution $u_{B,0}(x,t)$ to equation \eqref{eq1} whose exact expression is given by \eqref{expl0} in the Appendix. It is compactly supported for
$$
0<t<T:=\frac{1}{(\lambda_2-\lambda_1)^2}\tau_0,
$$
while its interface \emph{blows up to the left in finite time} at $t=T$ since $x=\log\,y/(\lambda_2-\lambda_1)\to-\infty$ as $y\to0$. Moreover, as $\tau\to\tau_0$, that is, $t\to T$, we have
\begin{equation*}
\begin{split}
u_{B,0}(x,t)&=e^{\lambda_1x/m}\theta(y,\tau)\sim Ke^{\lambda_1x/m}y^{1/(m-1)}\\
&=K\exp\left[\left(\frac{\lambda_2-\lambda_1}{m-1}+\frac{\lambda_1}{m}\right)x\right]
\end{split}
\end{equation*}
and since
$$
\frac{\lambda_2-\lambda_1}{m-1}+\frac{\lambda_1}{m}=-\frac{1}{3m+1}<0,
$$
we obtain that the solution $u_{B,0}$ to \eqref{eq1} is compactly supported for $0\leq t<T$ and both its interface and the function itself blow up in finite time at $t=T$ in the limit $x\to-\infty$. We plot in Figure \ref{fig1} in parallel the two explicit solutions $B_{C,y_0}$ and $u_{B,0}$ at different moments of time to show how the approaching and then focusing as $y\to0$, respectively blow-up as $x\to-\infty$ produce.
\begin{figure}[ht!]
  \begin{center}
  \subfigure[Explicit solution to equation \eqref{eq2}]{\includegraphics[width=7cm,height=5.5cm]{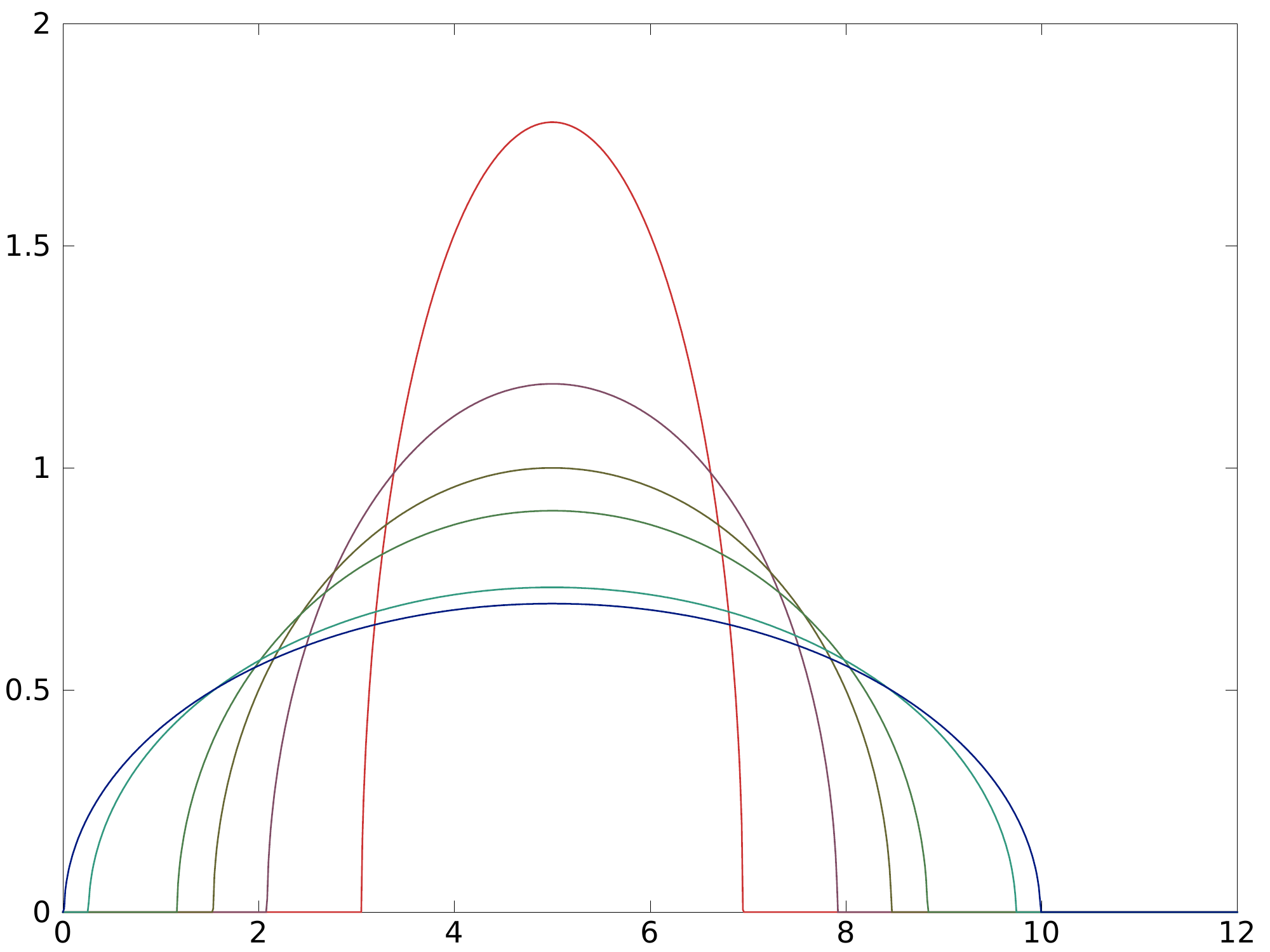}}
  \subfigure[Explicit solution to equation \eqref{eq1}]{\includegraphics[width=7cm,height=5.5cm]{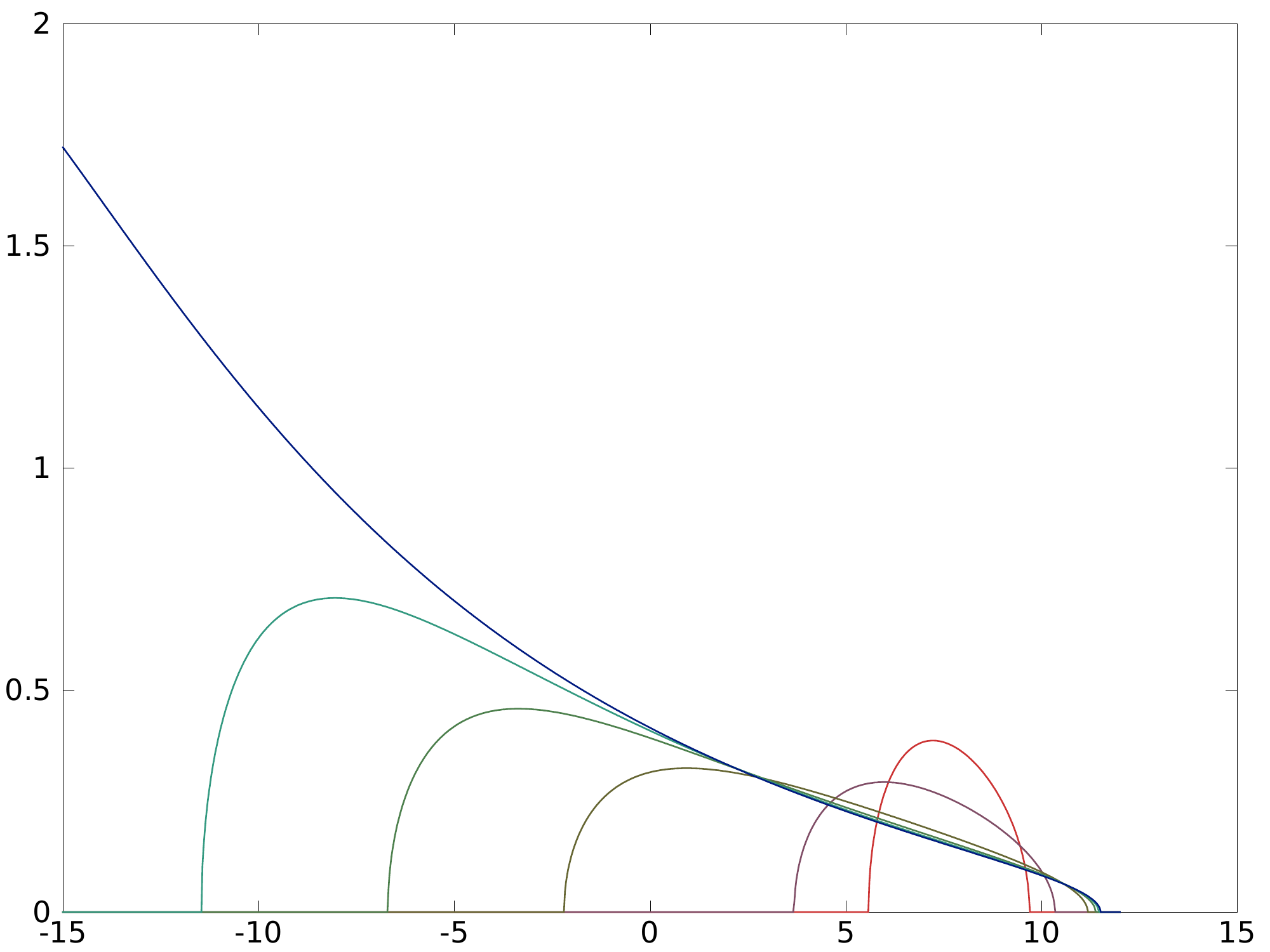}}
  \end{center}
  \caption{Explicit solutions presenting focusing as $y\to0$, respectively finite time blow-up as $x\to-\infty$. Experiment for $m=3$, $C=1$, $y_0=5$.}\label{fig1}
\end{figure}

The above phenomenon is in fact general. Let $\theta$ be a solution to \eqref{PME} with compactly supported initial condition $\theta_0$ such that ${\rm supp}\,\theta_0\subset(0,\infty)$. An easy consequence of \cite[Theorem 15.24, Section 15.5]{VPME} shows that (after a possible waiting time), $\theta$ behaves in a small inner neighborhood of the interface point (to the left) with the same rate as in \eqref{interm5}. Thus, the solution $u(x,t)$ transforming into $\theta$ by \eqref{change.Ksmall} is compactly supported for $0<t<T$ (where $T$ comes from the time required for the left interface of $\theta$ to "fill the hole" and reach the origin) and blows up in finite time, together with its interface, at $t=T$ and $x\to-\infty$.

\medskip

\noindent \textbf{Case 2. $0<K<2m(m+1)/(3m+1)^2$.} According to \eqref{change.Ksmall} and Remark 2 in Section \ref{sec.transf}, this range corresponds to equation \eqref{eq2.Ksmall} with $\gamma\in(0,(m+1)/m)$. For this range of $\gamma$, we furthermore use the transformation in \cite[Section 2.1, Case 2]{IRS13} which writes
\begin{equation}\label{interm6}
\theta(y,\tau)=y^{1/m}w(y^{\mu},\mu^2\tau), \quad \mu=\frac{(1-\gamma)m+1}{2m}>0,
\end{equation}
which maps \eqref{eq2.Ksmall} into the standard porous medium equation in radial variable
\begin{equation}\label{radial.PME}
w_s=(w^{m})_{rr}+\frac{N-1}{r}(w^m)_r, \quad r=y^{\mu}, \ s=\mu^2\tau, \quad N=2+\frac{1}{\mu}.
\end{equation}
For any solution $w$ to \eqref{radial.PME} with radially symmetric initial condition $w_0$ such that ${\rm supp}\,w_0\subseteq B(0,r_2)\setminus B(0,r_1)$ for some $r_1<r_2$, it is well known that the solution $w(r,s)$ converges asymptotically, according to \cite[Proposition 19.13, Chapter 19]{VPME}, to the self-similar focusing profile introduced by Aronson and Graveleau \cite{AG93} (see also \cite[Section 19.2]{VPME}). In particular, for any $s>0$ there exist $r_1(s)<r_2(s)$ such that ${\rm supp}\,w(s)\subseteq B(0,r_2(s))\setminus B(0,r_1(s))$ and there exists $s_0>0$ such that $r_1(s_0)=0$. It is also known that when $s\to s_0$, the interface behavior is given by
\begin{equation}\label{interm7}
w(r,s)\sim r^{(2\beta_*-1)/(m-1)\beta_*},
\end{equation}
where $\beta_*\in(1/2,1)$ is the self-similar anomalous focusing exponent \cite[Section 19.2, p. 459]{VPME}. Combining \eqref{interm6} and \eqref{interm7} and taking into account that $\mu>0$, we find that the solution $\theta$ to \eqref{eq2.Ksmall} has a left interface at $y=0$ for $\tau=\tau_0=s_0/\mu^2$ and near the interface, \eqref{interm6} gives the local behavior
\begin{equation}\label{interm8}
\theta(y,\tau)\sim y^{1/m}y^{(2\beta_*-1)\mu/(m-1)\beta_*}=y^{\phi(\gamma,\beta_{*})},
\end{equation}
with
$$
\phi(\gamma,\beta_*)=\frac{1}{m}+\frac{(2\beta_*-1)(m+1-m\gamma)}{2m(m-1)\beta_*}.
$$
Passing now to \eqref{eq1} through the transformation \eqref{change.Ksmall}, we find that there exists
$$
T=\frac{\tau_0}{(\lambda_2-\lambda_1)^2}
$$
such that $u(x,t)$ is compactly supported for $0\leq t<T$ and its interface blows up to the left as $t\to T$. We furthermore infer from \eqref{interm8} and \eqref{change.Ksmall} that
$$
u(x,t)\sim e^{\lambda_1x/m}y^{\phi(\gamma,\beta_*)}\sim\exp\left[\left(\frac{\lambda_1}{m}x+(\lambda_2-\lambda_1)\phi(\gamma,\beta_*)\right)x\right].
$$
Let us notice first that, since $1/2<\beta_*<1$, we have $0<(2\beta_*-1)/\beta_*<1$. Recalling also the values of $\lambda_1$ and $\lambda_2$ from \eqref{lambdas} and the formula for $K$ with respect to $m$ and $\gamma$ in \eqref{interm9}, we have
$$
\sqrt{1-4K}=\frac{m-1}{|2m\gamma-(3m+1)|}=\frac{m-1}{3m+1-2m\gamma},
$$
and we further obtain
\begin{equation*}
\begin{split}
\frac{\lambda_1}{m}+(\lambda_2-\lambda_1)\phi(\gamma,\beta_*)&=\frac{\lambda_1}{m}+\frac{\lambda_2-\lambda_1}{m}+
\frac{(\lambda_2-\lambda_1)(2\beta_*-1)(m+1-m\gamma)}{2m(m-1)\beta_*}\\
&\leq\frac{\lambda_2}{m}+\frac{(\lambda_2-\lambda_1)(m+1-m\gamma)}{2m(m-1)}\\
&=\frac{1}{m}\left[-\frac{1}{2}+\frac{\sqrt{1-4K}}{2}+\frac{(m+1-m\gamma)\sqrt{1-4K}}{2(m-1)}\right]\\
&=\frac{m\gamma-(m+1)}{3m+1-2m\gamma}<0,
\end{split}
\end{equation*}
recalling that $0<\gamma<\frac{m+1}{m}$. It follows that the solution $u$ blows up in finite time at $t=T$ and at $x=-\infty$, similarly to Case 1.

\medskip

\noindent \textbf{Case 3. $2m(m+1)/(3m+1)^2<K<1/4$.} According to \eqref{change.Ksmall} and Remark 2 in Section \ref{sec.transf}, this range corresponds to equation \eqref{eq2.Ksmall} with $\gamma\in(-\infty,0)$. In such case, the nonhomogeneous equation \eqref{eq2.Ksmall} has not been deeply studied, but a simple inspection shows that the transformation \eqref{interm6} above, which maps \eqref{eq2.Ksmall} into the standard porous medium equation in radial variable \eqref{radial.PME}, works similarly as in Case 2. Thus, we can again start from the radial porous medium equation \eqref{radial.PME} and repeat the proof in Case 2 above to show that finite time blow-up (both of the interface and of the solution itself) occurs as $x\to-\infty$.

\subsection{Proof of Theorem \ref{th.bu} when $K=1/4$}\label{sec.Kequal}

In this part we deal with equation \eqref{eq1} with $K=1/4$. As explained in the Introduction, this number also appears as a limit for the existence of the traveling waves in \cite{dPS00}, thus it seems to be a critical coefficient for the equation. Recall that in this case, the change of variable \eqref{change.Kequal} maps solutions to equation \eqref{eq1} into solutions to the nonhomogeneous equation with exponential density \eqref{eq2.Kequal}. Our starting point will be in this case the study of equation \eqref{eq2.Kequal} performed in \cite[Section 4]{GK96}. According to it, equation \eqref{eq2.Kequal} has solutions of the form
\begin{equation}\label{sols.Kequal}
\theta(y,\tau)=(T-\tau)^{\gamma}f(\eta), \quad \eta=y+\lambda\ln\,(T-\tau), \ \gamma=\frac{\lambda-1}{m-1}>0,
\end{equation}
with $\lambda>1$ and presenting blow-up of the interface at $\tau=T$. The profiles $f(\eta)$ satisfy one of the following two behaviors as $\eta\to-\infty$:
\begin{equation}\label{beh1}
f(\eta)\sim\exp\left(-\frac{\lambda-1}{(m-1)\lambda}\eta\right), \quad {\rm as} \ \eta\to-\infty,
\end{equation}
or
\begin{equation}\label{beh2}
f(\eta)\sim\left(\frac{m-1}{m^2}\right)^{1/(m-1)}e^{-\eta/(m-1)}, \quad {\rm as} \ \eta\to-\infty.
\end{equation}
Moreover, it is proved in \cite{GK96} that there exists a unique $\lambda\in(1,2)$ such that there exists a unique profile with behavior as in \eqref{beh1}, for that particular (unique) value of $\lambda$, having an interface to the right. We transform these profiles to solutions to equation \eqref{eq1} by using \eqref{change.Kequal} to get
\begin{equation}\label{sol.Kequal}
u(x,t)=\left(\frac{m-1}{2m}\right)^{2\gamma}(\overline{T}-t)^{\gamma}e^{-x/2m}f\left(-\frac{m-1}{m}x+2\lambda\ln\frac{m-1}{2m}+\lambda\ln(\overline{T}-t)\right),
\end{equation}
where $\overline{T}=4m^2T/(m-1)^2$. Let us notice that the limit process $\eta\to-\infty$ changes into $x\to+\infty$ and that an interface to the right for the solution $\theta$ to equation \eqref{eq2.Kequal} is mapped into an interface to the left for the solution $u$ to equation \eqref{eq1} given by \eqref{sol.Kequal}. We divide the rest of our analysis into two cases according to the behaviors in \eqref{beh1} and \eqref{beh2}.

\medskip

\noindent \textbf{Case 1.} If we start from the unique profile $f(\eta)$ with behavior as in \eqref{beh1}, we obtain that as $x\to\infty$
\begin{equation*}
\begin{split}
u(x,t)&\sim C(\overline{T}-t)^{\gamma}\exp\left[-\frac{x}{2m}-\frac{\lambda-1}{\lambda(m-1)}\left(-\frac{m-1}{m}x+C+\lambda\ln(\overline{T}-t)\right)\right]\\
&=C(\overline{T}-t)^{\gamma}(\overline{T}-t)^{-(\lambda-1)/(m-1)}\exp\left[-\frac{x}{2m}+\frac{\lambda-1}{m\lambda}x\right]\\
&=C\exp\left[\frac{(\lambda-2)x}{2m\lambda}\right]\to0,
\end{split}
\end{equation*}
where we recall that $1<\lambda<2$ and by definition $\gamma=(\lambda-1)/(m-1)$, and we denote by $C>0$ a generic constant. We thus find that the solution $u$ has an exponentially decreasing tail as $x\to+\infty$. At the other end, we know that $u(x,t)$ present an interface to the left at some point
$$
x_0=-\frac{m-1}{m}x+C+\lambda\ln(\overline{T}-t),
$$
thus
$$
x=\frac{m\lambda}{m-1}\ln(\overline{T}-t)-\frac{m}{m-1}(x_0+C)\to-\infty, \quad {\rm as} \ t\to\overline{T},
$$
which proves the finite time blow-up of the interface. Due to the compensation of the terms in $(\overline{T}-t)$ in the calculation above, we also remark that for any $x\in\real$ fixed, $u(x,t)$ remains bounded for any $t>0$. We prove next that the solution $u$ blows up in finite time (at $t=\overline{T}$) as $x\to-\infty$. To this end, let us consider the \emph{self-similar directon} of the profile $f(\eta)$, that is, when $\eta=C$ constant, meaning
$$
x=\frac{m\lambda}{m-1}\ln(\overline{T}-t)-C
$$
and
\begin{equation}\label{interm10}
\begin{split}
u(x,t)&\sim C(\overline{T}-t)^{\gamma}e^{-x/2m}=C(\overline{T}-t)^{\gamma-\lambda/2(m-1)}\\
&=C(\overline{T}-t)^{(\lambda-2)/2(m-1)},
\end{split}
\end{equation}
which blows up as $t\to\overline{T}$, since $\lambda<2$. The latter does not only prove the finite time blow-up, but also shows that the blow-up rate is equal or faster than the rate given by $(\overline{T}-t)^{(\lambda-2)/2(m-1)}$.

\medskip

\noindent \textbf{Case 2.} If we start from a profile $f(\eta)$ with behavior as in \eqref{beh2}, we obtain that as $x\to\infty$
\begin{equation*}
\begin{split}
u(x,t)&\sim C(\overline{T}-t)^{\gamma}\exp\left[-\frac{x}{2m}-\frac{1}{m-1}\left(-\frac{m-1}{m}x+C+\lambda\ln(\overline{T}-t)\right)\right]\\
&=Ce^{x/2m}(\overline{T}-t)^{-1/(m-1)},
\end{split}
\end{equation*}
thus we find solutions such that $u(x,t)\to\infty$ as $x\to\infty$ and that blow up globally at $t=\overline{T}$. These solutions have an interface to the left that also blows up at $t=\overline{T}$.

\medskip

\noindent Using the large time behavior for any compactly supported solution to \eqref{eq2.Kequal} \cite[Theorem 3]{GK96} together with the transformation \eqref{change.Kequal}, we readily deduce that \emph{any general solution} $u(x,t)$ to equation \eqref{eq1} with $K=1/4$ with compactly supported initial condition blows up in finite time as $x\to-\infty$ exactly in the same way as described in Case 1 above, that is, with a blow-up rate equal or faster than the rate given by $(\overline{T}-t)^{(\lambda-2)/2(m-1)}$. This ends the proof of Theorem \ref{th.bu}.

\section{Case $K>1/4$. Large time behavior for a homogeneous Dirichlet problem}\label{sec.Klarge}

For $K>1/4$, we recall that through the transformation \eqref{change.Kbig}, we arrive to a nonhomogeneous equation of the form \eqref{eq2.Kbig} with the density function given by the more involved formula (with respect to the previous cases) \eqref{interm4}. In particular, we have to restrict ourselves to working in the bounded interval
\begin{equation}\label{interv.Dir}
x\in\left[-\frac{\pi}{\sqrt{4K-1}},\frac{\pi}{\sqrt{4K-1}}\right].
\end{equation}
It is established in \cite[Theorems 1 and 2]{GK96} that any solution to equation \eqref{eq2.Kbig} with compactly supported initial condition $\theta_0(y)=\theta(y,0)$ presents finite time blow-up of the interface at some time $t=T\in(0,\infty)$ and then can be continued for $t>T$ with the property $\lim\limits_{|y|\to\infty}\theta(y,t)=0$ for any $t>T$. Using our transformation \eqref{change.Kbig} we find
\begin{equation}\label{sol.Kbig}
u(x,t)=e^{-x/2m}\cos\left(\frac{\sqrt{4K-1}}{2}x\right)^{1/m}\theta(y,t)\to0, \quad {\rm as} \ x\to\pm\frac{\pi}{\sqrt{4K-1}},
\end{equation}
for any $t>0$. Thus, in the case $K>1/4$ our transformation gives information about the solutions to the homogeneous Dirichlet problem posed in the bounded interval \eqref{interv.Dir}. Since the density function satisfies
$$
f(y)\sim|y|^{-(3m+1)/m}, \quad {\rm as} \ |y|\to\infty,
$$
it follows that $f\in L^{1}(\real)$, thus it is well known that equation \eqref{eq2.Kbig} satisfies the \emph{property of isothermalization} \cite{KR82}: for any solution $\theta$ with a compactly supported initial condition $\theta_0(y)=\theta(y,0)$, $y\in\real$, letting
\begin{equation}\label{iso}
\overline{\theta}=\left[\int_{-\infty}^{\infty}f(y)\,dy\right]^{-1}\int_{-\infty}^{\infty}\theta_0(y)f(y)\,dy,
\end{equation}
then it is proved in \cite[Theorem 1]{KR82} that $\theta(y,t)\to\overline{\theta}$ as $t\to\infty$, with uniform convergence in compact sets. We may thus conclude from \eqref{iso} and \eqref{sol.Kbig} that
\begin{equation}\label{iso.eq1}
\lim\limits_{t\to\infty}u(x,t)=e^{-x/2m}\cos\left(\frac{\sqrt{4K-1}}{2}x\right)^{1/m}\overline{\theta},
\end{equation}
for any solution $u$ to the homogeneous Dirichlet problem for equation \eqref{eq1} in the interval \eqref{interv.Dir}, with initial condition $u_0$ such that
$$
{\rm supp}\,u_0\subset\left(-\frac{\pi}{\sqrt{4K-1}},\frac{\pi}{\sqrt{4K-1}}\right)
$$
and uniform convergence over compact sets inside the interval \eqref{interv.Dir}.

\medskip

\noindent \textbf{Remark.} The large time behavior given in \eqref{iso.eq1} extends the convergence result established for equations without convection terms in \cite[Remark, Section 7.2, p. 436]{S4}. In the latter case the limit profile is symmetrical, while in our case the effect of the convection term brings the presence of the factor $e^{-x/2m}$ breaking the symmetry of the limit function.

\section{An application for the nonhomogeneous porous medium equation}\label{sec.app}

In the previous sections of the paper, we have exploited the transformation introduced in Section \ref{sec.transf} and the general knowledge related to equation \eqref{eq2} with the corresponding density functions in order to obtain new results on the solutions to equation \eqref{eq1}. But the transformation works in both directions, and in the current section we apply it in order to complete the theory of the nonhomogeneous porous medium equation equation \eqref{eq2}. More precisely, we prove the following result.
\begin{theorem}\label{th.nohom}
Let $\theta$ be a solution to equation \eqref{eq2} posed for $y\in(0,\infty)$ with density $f(y)=y^{-\gamma}$, with $\gamma>2$ and such that its initial condition
$$
\theta_0(y)=\theta(y,0), \quad y\in[0,\infty)
$$
is continuous and compactly supported in $(0,\infty)$. Then the interface blows up in finite time to the right: there exists $T>0$ such that ${\rm supp}\,\theta(y,\tau)$ is compact in $(0,\infty)$ for any $\tau\in(0,T)$ but there exists $R>0$ such that $\theta(y,T)>0$ for any $y>R$.
\end{theorem}
\begin{proof}
This is now an easy consequence of the previous sections. Let us start from equation \eqref{eq1} with $K\in(0,1/4)$. Using the results in Section \ref{sec.Ksmall} and making the change of variable $x\mapsto-x$, we readily get that if $u$ is a solution to
\begin{equation}\label{eq1.minus}
u_t=(u^m)_{xx}-(u^m)_x+Ku^m, \quad 0<K<1/4,
\end{equation}
such that $\supp\,u_0(x)$ is compact, then there exists $T\in(0,\infty)$ such that $u$ blows up at $t=T$ as $x\to\infty$ (both the function and its interface). We then notice that we can use the same transformation in Section \ref{sec.transf} to reach equation \eqref{eq2} with $f(y)=y^{-\gamma}$, where
$$
\gamma=\frac{1}{2}\left(\frac{3m+1}{m}+\frac{m-1}{m\sqrt{1-4K}}\right).
$$
We thus infer that (with $m>1$ fixed) the range $K\in(0,1/4)$ is mapped into $\gamma\in(2,\infty)$ and a compactly supported initial condition $u_0$ to equation \eqref{eq1.minus} (and also to equation \eqref{eq1}) is mapped into an initial condition $\theta_0$ to equation \eqref{eq2} with compact support included in the open interval $(0,\infty)$. The blow-up of the interface for solutions to \eqref{eq1.minus} leads to the conclusion.
\end{proof}

\noindent \textbf{Remark.} Comparing equation \eqref{eq2} with $f(y)=y^{-\gamma}$ with on the one hand $\gamma\in(0,2)$ and on the other hand $\gamma>2$, we notice that if we start at $\tau=0$ with an initial condition $\theta_0(y)$ with compact support inside the open interval $(0,\infty)$, for $\gamma\in(0,2)$ we have \emph{finite time focusing} (that is, the interface to the left reaches the origin in finite time), while for $\gamma\in(2,\infty)$ we have finite time interface blow-up (that is, the interface to the right reaches $+\infty$ in finite time). This is a kind of "symmetry" with respect to $y=1$ that can be better understood by using the self-maps introduced in \cite{IRS13}.

\section{Appendix. Some explicit solutions, extensions and open problems.}

In this final section we gather some facts related to our two equations and the mapping between their solutions, such as some explicit solutions, some extensions of the transformation and a few open questions.

\medskip

\noindent \textbf{Some explicit solutions to equation \eqref{eq1}.} We start with the translated Barenblatt-type solution to the standard porous medium equation
$$
B_{C,y_0}(y,\tau)=\tau^{-1/(m+1)}\left[C-k\frac{(y-y_0)^2}{\tau^{2/(m+1)}})\right]_{+}^{1/(m-1)}, \quad k=\frac{m-1}{2m(m+1)},
$$
which is supported in a compact interval centered on $y_0>0$. Applying the change of variable \eqref{change.Ksmall} for $\gamma=0$ and the particular case
$K=2m(m+1)/(3m+1)^2$ and letting
$$
x_0=\frac{1}{\lambda_2-\lambda_1}\ln\,y_0, \quad \lambda=\lambda_2-\lambda_1
$$
where $\lambda_1$, $\lambda_2$ are as in \eqref{lambdas}, we obtain the following explicit solution to equation \eqref{eq1} with $K=2m(m+1)/(3m+1)^2$:
\begin{equation}\label{expl0}
u_{B,0}(x,t)=\lambda^{-2/(m+1)}t^{-1/(m+1)}e^{\lambda_1 x/m}\left[C-k\left(\frac{e^{\lambda x}-e^{\lambda x_0}}{t^{1/(m+1)}\lambda^{2/(m+1)}}\right)^{2}\right]_+^{1/(m-1)},
\end{equation}
which is compactly supported for small times $t>0$. We thus notice that the solution $u_{B,0}$ blows up as $x\to-\infty$ together with its interface and its blow-up time is given by
$$
T^{1/(m+1)}=\frac{e^{\lambda x_0}}{\lambda^{2/(m+1)}\sqrt{C/k}}.
$$

For $\gamma>0$, there exist the explicit Barenblatt-type solutions to equation \eqref{eq2} with $f(y)=y^{-\gamma}$ and $-\infty<\gamma<(m+1)/m$, which write \cite{KRV10, IRS13}
\begin{equation}\label{Bar}
B_{\gamma}(y,\tau)=\tau^{-\alpha}\left[C-k\left(\frac{y}{\tau^{\beta}}\right)^{2-\gamma}\right]_{+}^{1/(m-1)},
\end{equation}
where $C>0$ is a free constant and
\begin{equation}\label{exp}
\alpha=\frac{1-\gamma}{m+1-m\gamma}, \ \ \beta=\frac{1}{m+1-m\gamma}, \ \ k=\frac{m-1}{m(2-\gamma)(m+1-m\gamma)}.
\end{equation}
Notice that these solutions cannot be translated (as we did above for $\gamma=0$) since the equation equation \eqref{eq2} is no longer invariant to translations. Applying the change of variable \eqref{change.Ksmall} to \eqref{Bar} and letting again $\lambda=\lambda_2-\lambda_1$ we find after straightforward calculations the following explicit solution to equation \eqref{eq1} with $K\in(0,1/4)$
\begin{equation}\label{expl1}
u_B(x,t)=\lambda^{-2\alpha}t^{-\alpha}e^{\lambda_1 x/m}\left[C-k\left(\frac{e^{\lambda x}}{t^{\beta}\lambda^{2\beta}}\right)^{2-\gamma}\right]_+^{1/(m-1)},
\end{equation}
where $\alpha$, $\beta$, $\gamma$ are given by \eqref{exp} and $\lambda_1$, $\lambda_2$ are given by \eqref{lambdas}. This explicit solution (for any fixed $C>0$) has an interface to the right at any time $t>0$ and satisfies
$$
\lim\limits_{x\to-\infty}u_B(x,t)=+\infty, \quad {\rm for \ any} \ t>0.
$$
This is why this family of solutions does not enter in the framework of Theorem \ref{th.bu}: they have no interface to the left and moreover, they are already "blown up" from the beginning at $-\infty$. However, these solutions can be used for comparison from above to give an alternative proof of the fact that solutions to equation \eqref{eq1} as in the statement of Theorem \ref{th.bu} remain bounded forever at any fixed point $x\in\real$.

Another interesting explicit solution to \eqref{eq2} for $f(y)=y^{-\gamma}$ and $-\infty<\gamma<(m+1)/m$ is the dipole-type solution given by \cite{KR81, IRS13}
\begin{equation}\label{dip}
Z_{\gamma}(y,\tau)=\tau^{-(m+1-m\gamma)/m^2(2-\gamma)}y^{1/m}\left[C-k\left(\frac{y}{\tau^{1/m(2-\gamma)}}\right)^{(m+1-m\gamma)/m}\right]_{+}^{1/(m-1)}
\end{equation}
where $C>0$ is a free constant and $k$ is given in \eqref{exp}. Applying the change of variable \eqref{change.Ksmall} to \eqref{dip}, we find after straightforward calculations the following explicit solution to equation \eqref{eq1} with $K\in(0,1/4)$
\begin{equation}\label{expl2}
u_Z(x,t)=\lambda^{-2\nu}t^{-\nu}e^{\lambda_2 x/m}\left[C-k\left(\frac{e^{\lambda x}}{t^{1/\omega}\lambda^{2/\omega}}\right)^{(m+1-m\gamma)/m}\right]_+^{1/(m-1)},
\end{equation}
where $\nu=(m+1-m\gamma)/m(2-\gamma)>0$, $\omega=m(2-\gamma)$, $\lambda=\lambda_2-\lambda_1$ and again $\lambda_1$, $\lambda_2$ are given by \eqref{lambdas}. Since for $K\in(0,1/4)$ we have $\lambda_2<0$, this family of solutions has similar properties as the previous family obtained from the Barenblatt solutions: they have an interface to the right and
$$
\lim\limits_{x\to-\infty}u_Z(x,t)=+\infty, \quad {\rm for \ any} \ t>0.
$$

\medskip

\noindent \textbf{An extension of the transformation.} Our transformation in Section \ref{sec.transf} can be still extended to a slightly more general class of equations, namely
\begin{equation}\label{eq1.verygen}
w_s=(w^m)_{xx}+a(x)(w^m)_{x}+b(x)w^m+c(s)w, \quad (x,s)\in\real\times(0,\infty).
\end{equation}
Indeed, a solution $w$ to equation \eqref{eq1.verygen} can be mapped onto a solution $u$ to equation \eqref{eq1.gen} by the following mapping
$$
w(x,s)=f(s)u(x,t), \quad t=g(s),
$$
where $f$, $g$ are obtained as solutions to the differential equations $f'(s)=c(s)f(s)$ and $g'(s)=f^{m-1}(s)$. Joining this change of variable with the transformation in Section \ref{sec.transf}, we can also transform solutions to equation \eqref{eq1.verygen}.

\medskip

\noindent \textbf{Some open questions.} As we have seen, the transformation in Section \ref{sec.transf} has a serious technical limitation for $K>1/4$, where it is unable to map solutions whose support is not included inside the compact interval \eqref{interv.Dir}. Thus, it is a natural open question to ask whether blow-up in finite time for solutions to equation \eqref{eq1} occurs in the same way as described in Theorem \ref{th.bu}. We already know that on the one hand finite time blow-up must occur as proved in \cite{Su98} for any coefficient $K>0$, but on the other hand it seems that $K=1/4$ is not just a technical limitation but a critical coefficient for equation \eqref{eq1}, as also establised in \cite{dPS00}. Thus, it would be interesting to study closely the blow-up phenomenon for solutions to equation \eqref{eq1} also when $K>1/4$. Moreover, another natural open question is to extend the results to more general reaction-convection-diffusion with exponents different from $m$, that is
$$
u_t=(u^m)_{xx}+(u^q)_x+Ku^p,
$$
for different values of $m$, $q$ and $p$. In such cases it appears that there is no transformation mapping it onto another (better studied) equation, but at least for the range of exponents
$$
q>m-1, \quad \max\{m,q\}\leq p<\min\{m+2,q+1\}
$$
for which it is proved in \cite{Su98} that all the solutions blow up in finite time, one might expect that the result of Theorem \ref{th.bu} still holds true, as there is no particular reason for a change in the way how blow-up occurs. Techniques directly related to equation \eqref{eq1} should be used.

\section*{Acknowledgements} B. H.-B. acknowledges Ministerio de Econom\'ia, Industria y Competitividad for grants MTM2017-84383-P and MTM2016-80276-P. The work of P. R. G. and A. P. is supported by the Ministry of Economy and Competitiveness of Spain under contract MTM2016-80276-P (AEI/FEDER, EU). The work of A. S. is partially supported by the Spanish project MTM2017-87596-P.

\bibliographystyle{plain}

\end{document}